\documentclass[10pt]{amsart}
\usepackage{amsthm}
\usepackage{amsmath}
\usepackage{color,soul}
\usepackage{todonotes}
\usepackage{amssymb}
\usepackage{circledsteps}
\theoremstyle{plain}
\newtheorem{theorem}{Theorem}
\newtheorem*{mainthm}{Main Theorem}
\theoremstyle{definition}
\newtheorem{fact}{Fact}

\newtheorem{proposition}{Proposition}

\newtheorem{lemma}[proposition]{Lemma}
\newtheorem{definition}[proposition]{Definition}
\theoremstyle{remark}
\newtheorem*{claim}{Claim}

\DeclareMathOperator{\row}{Row}
\DeclareMathOperator{\col}{Col}

\DeclareMathOperator{\tcf}{tcf}

\newcommand{\pcfsig}{\text{\rm pcf}_{\sigma\text{\rm-com}}}

\DeclareMathOperator{\cf}{cf}

\DeclareMathOperator{\pcf}{pcf}

\DeclareMathOperator{\Reg}{{\sf REG}}
\newcommand{\sk}{\vskip.05in}

\newcommand{\restr}{\upharpoonright}

\DeclareMathOperator{\reg}{{\sf Reg}}

\DeclareMathOperator{\pp}{pp}

\numberwithin{equation}{section}

\begin{document}
\title{A note on the Revised GCH}
\begin{abstract}
We give a proof Theorem~2.10 from \cite{460} that eliminates the use of Shelah's nice filters and associated rank functions, and instead uses only the well-foundedness of reduced products of ordinals modulo countably complete filters.  This removes any need to make assumptions about the existence of large cardinals in core models.
\end{abstract}
\author{Todd Eisworth}
\email{eisworth@ohio.edu}
\thispagestyle{empty}
\maketitle
\section{Introduction}

Shelah's Revised GCH is one of the jewels of pcf theory.  It is most commonly stated in terms of a ``revised power set'' operation $\lambda^{[\kappa]}$, defined for regular $\kappa<\lambda$ by letting $\lambda^{[\kappa]}$ be the minimal cardinality of a family $\mathcal{P}\subseteq[\lambda]^{\leq\kappa}$ such that every subset of $\lambda$ of cardinality $\kappa$ is a union of fewer than $\kappa$ elements of $\mathcal{P}$.

\begin{theorem}[RGCH]
If $\mu$ is an uncountable strong limit cardinal, then for every $\lambda\geq\mu$ there is a $\kappa_0<\mu$ such that for every $\kappa_0\leq\kappa<\mu$,
\begin{equation}
\lambda^{[\kappa]}=\lambda.
\end{equation}
\end{theorem}

Shelah has given several proofs of this result, with three of them appearing in the original paper \cite{460} itself.  Other proofs appear in \cite{513} and \cite{829}, and the proof from the latter paper has a nice explication in the Handbook of Set Theory article on pcf theory by Abraham and Magidor \cite{AM}.  One of Shelah's proofs actually establishes a more general theorem, weakening the hypothesis on $\mu$ to a statement in pcf theory (Theorem~2.10 of~\cite{460}), with the full revised GCH being an easy corollary if $\mu$ is assumed to be a strong limit.  His proof makes use of so-called ``nice filters'' and their associated ranks (as presented in the series of papers \cite{111,256,420} as well as Chapters V and VI of~\cite{cardarith}).  This technology works well, but it does rely on assumptions about the existence of Ramsey cardinals in suitable versions of the core model. This is not an unreasonable assumption (it will hold in situations where the cardinal exponentiation of singular cardinals is ``interesting''), but the point of this paper is to show that such an assumption is unnecessary and that we can push through a proof of his result using only the standard Galvin-Hajnal rank associated with countably complete filters:  all we need to use is that reduced products of ordinals modulo a countably complete filter are well-founded, so that any non-empty collection of equivalence classes will have a minimal element.

The theorem we prove is as follows:

\begin{mainthm} Suppose $\mu$ is a singular cardinal such that whenever $A$ is a set of regular cardinals satisfying
$|A|<\mu<\min(A)$, we have $|\pcf(A)|<\mu$.
Then for every $\lambda>\mu$ there is a $\sigma<\mu$ such that
\begin{equation}
A\subseteq (\mu,\lambda)\cap\reg\text{ and }|A|<\mu\Longrightarrow \pcfsig(A)\subseteq\lambda.
\end{equation}
\end{mainthm}

Note that assumption will hold if $\mu$ is a strong limit, as the cardinality of $\pcf(A)$ is always at most $2^{|A|}$.  The conclusion of the original version of RGCH follows from the strong limit assumption as well, using some standard pcf arguments as in the introduction of~\cite{460}.

Note that if $\cf(\mu)<\kappa=\cf\kappa<\mu$, then there must exist a $\theta<\mu$ such that $|\pcf(A)|<\theta$ whenever $A$ is a set of regular cardinals beyond $\mu$ of cardinality less than $\kappa$.  This is easy to see: if it fails, then by taking unions of counterexamples, we can build a set $A$ with $|A|<\kappa$ such that $|\pcf(A)|\geq\mu$, contradicting our assumption on $\mu$.

Finally, we point out that the pcf-theoretic ideas we use to push through the proof are due to Shelah. Our contribution was to realize that the argument given for his third proof of the RGCH in~\cite{460} could be simplified, and then combined with ideas from his second proof to give the result stated here, eliminating the need to rely on the filters and ranks from Section~5 of \cite{420}.

\section{Vocabulary and background assumptions}
We assume that the reader is familiar with basic pcf theory through the level presented in~\cite{AM} or~\cite{3germans}.  In particular, we assume that the reader is familiar with properties of the pcf ideals $J_{<\lambda}[A]$ and their associated generators $B_\lambda[A]$ for $\lambda$ in $\pcf(A)$.  We also sometimes use the quantifier $\forall^J$ in the context of an ideal $J$ to abbreviate the expressions ``for $J$-almost all'',  and to indicate that something holds except possibly for a set of exceptions that are in the ideal $J$.

The proof we give hinges on the different ways in which cardinals might be represented as the true cofinality of products.  The following definition captures the ideas we need.

\begin{definition}\label{defn1}
Suppose $\lambda$ is a singular cardinal.
\begin{enumerate}
\item A pair $(\bar{\lambda},J)$ is said to be an admissible pair at $\lambda$ if
\begin{itemize}
\item $\bar{\lambda}=\langle\lambda_i:i<\cf(\lambda)\rangle$ is a sequence of regular cardinals with $\cf(\mu)<\lambda_i<\lambda$
\item $J$ is a proper $\cf(\lambda)$-complete ideal on $\cf(\lambda)$ extending the ideal of bounded sets
\item for each $\mu<\lambda$, $\{i<\kappa:\lambda_i\leq\mu\}\in J$
\end{itemize}
\item If $J$ is a fixed ideal on $\cf(\lambda)$, then we say that $\bar{\lambda}$ is $J$-admissible at $\lambda$ if $(\bar{\lambda},J)$ is an admissible pair at $\lambda$.
\item  An admissible pair $(\bar{\lambda}, J)$ at $\lambda$ is said to {\em represent} a cardinal $\chi$ at $\lambda$ if
    \begin{equation}
    \tcf\left(\prod_{i<\kappa}\lambda_i/ J\right)=\chi,
    \end{equation}
which means that  the reduced product has true cofinality $\chi$, that is, $\prod_{i<\kappa}\lambda_i$ contains a $<_J$-increasing and cofinal sequence of length~$\chi$.
\item We say a cardinal $\chi$ is representable at $\lambda$ if $\chi$ is represented by a suitable pair $(\bar{\lambda}, J)$ at $\lambda$.  We say that $\chi$ is $J$-representable at $\lambda$ if it is $(\bar{\lambda}, J)$-representable at $\lambda$ for some $J$-admissible $\bar{\lambda}$ at $\lambda$. Finally, we may say that $\chi$ is representable on cofinality $\kappa$ if it is representable at some singular cardinal of cofinality $\kappa$.
\end{enumerate}
\end{definition}

We are going to use two basic facts about the way cardinals can be represented by admissible pairs.  The first of these is one of Shelah's ``No Holes'' results that tell us the set of cardinals representable at $\lambda$ by admissible pairs using a fixed suitable ideal $J$ form an interval of regular cardinals.

\begin{fact}
Suppose $\lambda$ is a singular cardinal, and  $(\bar{\lambda}, J)$ is an admissible pair representing some regular cardinal $\chi$ at $\lambda$.  Then for every cardinal $\eta\in (\lambda,\chi)\cap\reg$, we can find a $J$-admissible sequence $\bar{\tau}$ such that
\begin{enumerate}
\item $(\bar{\tau}, J)$ represents $\eta$ at $\lambda$, and 

\sk

\item $(\forall^Ji<\cf(\lambda))[\tau_i<\lambda_i]$.
\end{enumerate}
\end{fact}
\begin{proof}
Given such a regular cardinal $\eta$, we know that $\prod_{i<\kappa}\lambda_i$ is $\eta$-directed modulo $J$
because in fact it is $\chi$-directed.  If $X$ is a $J$-positive subset of~$\kappa$, then we know  that $\prod_{i\in X}\lambda_i$
cannot have true cofinality $\tau$ modulo the ideal $J\restr X$ for the trivial reason that it must still have true cofinality
$\chi$, and therefore Claim~1.4(1) of Chapter II in \cite{cardarith} tells us there is a sequence $\langle \tau_i:i<\kappa\rangle$ as desired.
\end{proof}

The second basic fact we need connects representability back to $\pcf(A)$ for $A$ a progressive set of regular cardinals.

\begin{fact}
Suppose $A$ is a progressive set of regular cardinals, and let $\sigma\leq|A|$ be an infinite regular cardinal.  If $\chi$ is in $\pcfsig(A)$, then there is a singular cardinal $\lambda$ such that
\begin{equation}
\sigma\leq\cf(\lambda)\leq|A| <\lambda<\chi
\end{equation}
and an admissible pair $(\bar{\lambda}, J)$ at $\lambda$ such that
\begin{itemize}
\item each $\lambda_i$ is in $\pcf(A)\cap\lambda$, and
\item the pair $(\bar{\lambda}, J)$ represents $\chi$ and $\lambda$.
\end{itemize}
\end{fact}
\begin{proof}
This is Claim 2.7 in \cite{460}.
\end{proof}
(For the above fact, the non-trivial issue is the requirement that the completeness of the filter must match the cofinality of the singular cardinal involved. If we are content with a $\sigma$-complete filter on $\kappa\geq\sigma$ then things are much easier.)

\section{The main theorem}

We remind the reader of our goal:

\begin{mainthm} Suppose $\mu$ is a singular cardinal such that whenever $A$ is a set of regular cardinals satisfying
$|A|<\mu<\min(A)$, we have $|\pcf(A)|<\mu$.
Then for every $\lambda>\mu$ there is a $\sigma<\mu$ such that
\begin{equation}
A\subseteq (\mu,\lambda)\cap\reg\text{ and }|A|<\mu\Longrightarrow \pcfsig(A)\subseteq\lambda.
\end{equation}
\end{mainthm}

The proof of this theorem will occupy the remainder of the paper, which we try to break up in subsections in an intelligible way.

\subsection{The setup}

The proof is by contradiction, so let us assume the theorem fails for a cardinal $\mu$ satisfying the assumption of the theorem, and let $\lambda$ be the least counterexample for $\mu$.   Thus, we are assuming that the following two statements hold:

\bigskip

\noindent\Circled{1} If $\mu<\lambda'<\lambda$ then there is a regular cardinal $\sigma<\mu$ such that 
\begin{equation}
A\subseteq(\mu,\lambda')\cap\reg\text{ and }|A|<\mu\Longrightarrow\pcfsig(A)\subseteq\lambda',
\end{equation}

\medskip

\noindent and

\medskip

\noindent\Circled{2} For every regular $\sigma<\mu$ there is an $A\subseteq(\mu,\lambda)\cap\reg$ such that $|A|<\mu$ but
\begin{equation}
\label{3.3}
\pcfsig(A)\nsubseteq\lambda.
\end{equation}

It follows easily from these two assumptions that $\lambda$ is singular and $\cf(\lambda)=\cf(\mu)$, as the failure of this condition would allow us to find a suitable $\sigma$ that works for $\lambda$ as well.   The following lemma digs a little deeper into the situation. 

\begin{lemma}
\label{setup} Suppose $\sigma<\mu<\epsilon<\lambda$.  Then we can find a singular cardinal $\lambda^*$ and an admissible pair $(\bar{\lambda}, J)$ at $\lambda^*$ such that
\begin{itemize}
\item $\sigma\leq\cf(\lambda^*)<\mu<\epsilon<\lambda^*<\lambda$, and
\item $(\bar{\lambda}, J)$ is an admissible pair at $\lambda^*$ representing $\lambda^+$.
\end{itemize}
\end{lemma}
In other words, the cardinal $\lambda^+$ is representable on some cofinality $\kappa$ satisfying $\sigma\leq\kappa<\mu$ at some singular cardinal $\lambda^*$ such that $\epsilon<\lambda^*<\lambda$.

\begin{proof}
First, we may freely increase $\sigma$ below $\mu$ if needed, so without loss of generality we may assume that $\sigma$ is a regular cardinal satisfying $\cf(\mu)=\cf(\lambda)<\sigma<\mu$, and furthermore, by~\Circled{1}, that
\begin{equation}
\label{blah}
A\subseteq(\mu,\epsilon)\cap\Reg\text{ and }|A|<\mu\Longrightarrow\pcfsig(A)\subseteq\epsilon.
\end{equation}

It suffices to find a set $B\subseteq(\epsilon,\lambda)\cap\reg$ of cardinality at most $|A|$ such that
\begin{equation}
\lambda^+=\max\pcf(B),
\end{equation}
and
\begin{equation}
\lambda^+\in\pcfsig(B).
\end{equation}
If we are able to do this, then Fact~2 will give us what we need: if $\lambda^*$ is the resulting singular cardinal, then $\lambda^*$ cannot be greater than $\lambda$ because of the associated admissible representation of $\lambda^+$, and $\lambda^*$ cannot equal $\lambda$ because the cofinality of $\lambda^*$ is at least $\sigma$.

Now how do we find such a $B$? By \Circled{2}, there is an $A\subseteq(\mu,\lambda)\cap\reg$ such that $|A|<\mu$ but~(\ref{3.3}) holds. Define
\begin{equation}
\chi:=\min\left(\pcfsig(A)\setminus\lambda\right),
\end{equation}
and let $\xi\leq\lambda$ be minimal such that $A\cap\xi$ is not in the $\sigma$-complete ideal $I$ on $A$ generated by $J_{<\chi}[A]$.  
 
It follows that $\xi$ is a singular cardinal less than or equal to $\lambda$ such that
\begin{equation}
\cf(\lambda)<\sigma\leq\cf(\xi)\leq|A|<\mu,
\end{equation}
and so in particular, $\xi<\lambda$.  By~(\ref{blah}), we know $\epsilon<\xi$ as well.  Since $A\cap\epsilon\in I$, we can remove this piece of $A$ if necessary and just assume $\epsilon<\min(A)$. 

We now employ some standard material from Chapter~II of~\cite{cardarith} to find our set $B$.  The cardinal $\chi$ witnesses that $\pp_{\Gamma(|A|^+,\sigma)}(\xi)$ is greater than $\lambda^+$, so by the No Holes Theorem (Conclusion 2.3(3A) of Chapter II in~\cite{cardarith}) we know there is a set $B$ and an ideal $J'$ such that

\begin{itemize}
\item $B$ is a cofinal subset of $\xi\cap\reg$ of cardinality at most $|A|$,

\sk 

\item $J'$ is a $\sigma$-complete ideal on $B$ extending the bounded ideal, and

\sk 

\item $\lambda^+ = \tcf\left(\prod B / J'\right)$.
\end{itemize}
Since $\epsilon<\xi$ and $J'$ includes all initial segments of $B$, we can assume that $B\cap\xi = \emptyset$ and then finish as outlined above using \Circled{2}.
\end{proof}

\subsection{The main lemma}

We will eventually derive our contradiction by finding a singular $\lambda^*$ in the interval $(\mu,\lambda]$ of uncountable cofinality~$\kappa$ with the property that for some $\kappa$-complete proper ideal on $\kappa$ extending the bounded ideal, the set of $J$-admissible sequences that represent $\lambda^+$ modulo $J$ is non-empty, but without a $<_J$-minimal member.  Since $J$ is countably complete, this is an impossibility.   The following lemma contains most of the hard work necessary to arrive at the conclusion.

\begin{lemma}
\label{mainlemma}
Suppose $(\bar\lambda, J)$ is an admissible pair representing $\lambda^+$ at some singular cardinal $\lambda^*$ satisfying $\mu<\lambda^*<\lambda$.  If there is a single regular $\sigma<\mu$ such that
\begin{equation}
\label{lemass}
(\forall^Ji<\kappa)\left[A\subseteq (\mu,\lambda_i)\cap\reg\wedge |A|<\mu\Longrightarrow \pcfsig(A)\subseteq\lambda_i\right].
\end{equation}
(We say that $\sigma$ {\em works for} the sequence $\bar\lambda$.)
Then we can find another $J$-admissible sequence $\bar{\psi}$ representing $\lambda^+$ at $\lambda^*$ such that
\begin{equation}
(\forall^Ji<\kappa)[\psi_i<\lambda_i].
\end{equation}
\end{lemma}

This isn't quite enough to get us a contradiction as there need not be a single $\sigma$ that works for the new sequence $\bar{\psi}$.  We will take care of this problem later on, when we split $\kappa$ into disjoint pieces on which the preceding lemma will apply, and then combine the resulting sequences to find our contradiction.

\begin{proof}

Given $\lambda^*$ and $(\bar{\lambda}, J)$, let $\kappa = \cf(\lambda^*)$.   Since increasing $\sigma$ below $\mu$ maintains our assumption, we may assume that $\sigma$ is larger than $\kappa$.  As noted at the end of the first section, our assumption on $\mu$ lets us fix a regular cardinal $\theta<\mu$ so large that $\theta<\mu$ so large that
\begin{equation}
A\subseteq(\mu,\lambda)\cap\reg\text{ and } |A|<\sigma\Longrightarrow |\pcf(A)|<\theta.
\end{equation}

By Lemma~\ref{setup} (and by increasing $\theta$ if needed) we can find a singular cardinal $\chi^*$ of cofinality $\theta$ such that $\lambda^*<\chi^*<\lambda$, and $\lambda^+$ can be represented by an admissible pair $(\bar{\chi},I)$ at $\chi^*$.  Since $I$ extends the bounded ideal on $\theta$, we may assume that $\lambda^*$ is less than each $\chi_i$ as well.

Now the idea is to use Fact 1, and represent each $\chi_i$ by a $J$-admissible sequence at $\lambda^*$ that is less than our original sequence $\bar{\lambda}$ modulo $J$.  More precisely, for each $\zeta<\theta$ we find a $J$- admissible sequence $\bar{\tau}^\zeta = \langle
\tau^\zeta_i:i<\kappa\rangle$ at $\lambda^*$  such that
\begin{equation}
(\forall^Ji<\kappa)[\tau^\zeta_i<\lambda_i]
\end{equation}
and
\begin{equation}
\tcf\left(\prod_{i<\kappa}\tau^\zeta_i/J\right)=\chi_\zeta.
\end{equation}
Now define
\begin{equation}
A: = \{\tau^\zeta_i:i<\kappa\text{ and }\zeta<\theta\}.
\end{equation}
Note that the various $\tau^\zeta_i$ need not be distinct, but our set $A$ is just the union of the ranges of the sequences
$\bar{\tau}^\zeta$ for $\zeta<\theta$.

Let's take a look at how the pcf-structure of $A$ reflects to the index set $\kappa\times\theta$.  We want to think of $A$ as
being presented like a matrix (with possible repetitions among entries), so for each $\zeta<\theta$ we have a corresponding row
\begin{equation}
\row(\zeta) = \langle \tau^\zeta_i:i<\kappa\rangle,
\end{equation}
and for each $i<\kappa$ we have a corresponding column
\begin{equation}
\col(i) = \{\tau^\zeta_i:\zeta<\theta\}.
\end{equation}
Note a subtle distinction in the notation: it will be to our advantage to view the rows as being admissible sequences of regular cardinals, while thinking of the columns as just being sets of regular cardinals.

Our next step is a bit of housekeeping.  We are going to redefine some of $\tau^\zeta_i$ to a new value of $\mu^+$ with an eye
towards achieving the following:
\begin{enumerate}
\item $\row(\zeta)$ will remain an admissible sequence representing $\chi_\zeta$ at $\lambda^*$ modulo $J$, and

\sk

\item $\col(i)$ will be a subset of $[\mu^+,\lambda_i)\cap\reg$ when $\mu^+<\lambda_i$.

\end{enumerate}

We can do this by brute force:  the set of $i$ for which $\lambda_i = \mu^+$ is in $J$, and for each $\zeta<\theta$ the set of
$i$ for which $\lambda_i\leq\tau_i^\zeta$ is in $J$.  For convenience, we may assume $\tau^0_0 = \mu^+$ so that $\mu^+\in A$.

More housekeeping follows, and this time we things so that $\max\pcf(A)=\lambda^+$.   Since $|\pcf(A)|<\mu<\min(A)$,
we can fix a transitive collection of generators $\langle B_\psi:\psi\in\pcf(A)\rangle$ for $\pcf(A)$ (so the
generators are subsets of $\pcf(A)$ and not just $A$).  We may assume that $\mu^+=\min(A)\in B_{\lambda^+}$.

Each $\chi_\zeta$ is in $\pcf(A)$, and certainly $\lambda^+$ is in $\pcf(A)$ as well.  For $\zeta<\theta$, the generator
$B_{\chi_\zeta}$ contains $\tau_i^\zeta$ for $J$-almost all $i<\kappa$.  Similarly, the generator $B_{\lambda^+}$ contains
$\chi_\zeta$ for $I$-almost all $\zeta$. Putting these together, we see

\begin{equation}
(\forall^I\zeta<\theta)(\forall^J i<\kappa)\left[\tau^\zeta_i\in B_{\lambda^+}\right].
\end{equation}

Now choose $\zeta^*$ such that $\chi_{\zeta^*}\in B_{\lambda^+}$.  If $\zeta<\theta$ is a place where $\chi_\zeta\notin
B_{\lambda^+}$, we redefine $\chi_\zeta$ to be $\chi_{\zeta^*}$, and redefine $\tau^\zeta_i$ to be $\tau^{\zeta^*}_i$, so that
now row $\zeta$ will now represent $\chi_{\zeta^*}$.  Since this affects only an $I$-small set of rows, our sequence $\langle
\chi_\zeta:\zeta<\theta\rangle$ will still represent $\lambda^+$ modulo~$I$, and we have improved our situation to obtain

\begin{equation}
(\forall \zeta<\theta)(\forall^J i<\kappa)\left[\tau^\zeta_i\in B_{\lambda^+}\right].
\end{equation}

By the above, for each  $\zeta<\theta$, the set of $i<\kappa$ for which $\tau^\zeta_i$ is NOT in $B_{\lambda^+}$ is in the ideal~$J$.  We redefine $\tau^\zeta_i$ to be $\mu^+$ in that situation.  Again, this modifies row $\zeta$ on a
$J$-small set, so that $\row(\zeta)$ still has the required properties, and we've made sure that $\max\pcf(A) = \lambda^+$ as our
``new'' $A$ is a subset of the original $B_{\lambda^+}$.

To summarize our efforts, we have arranged that our collection $A$ is a set of regular cardinals in
$(\mu,\lambda^*)$ of cardinality less than $\mu$ such that $\max\pcf(A)=\lambda^+$, and such that $A$ can be organized into
$\theta$ rows and $\kappa$ columns such that
\begin{itemize}
\item for each $\zeta<\theta$, row $\zeta$ represents $\chi_\zeta$ modulo $J$.

\sk

\item for each $i<\kappa$, column $i$ is a subset of $\{\mu^+\}\cup [\mu^+,\lambda_i)\cap\reg$

\sk

\end{itemize}

Now look back at our assumption~(\ref{lemass}).  If $i<\kappa$ indexes a column for which $\lambda_i>\mu^+$ and $\sigma$ ``works for'' $\lambda_i$ (which the case for almost all columns modulo
$J$), then because $\col(i)$ is a subset of $(\mu,\lambda_i)\cap\reg$ of cardinality less than $\mu$ we have
\begin{equation}
\pcfsig(\col(i))\subseteq (\mu,\lambda_i)\cap\reg.
\end{equation}
By the $\sigma$-complete version of pcf-compactness (see \cite{430} Claim 6.7F (2)), there is a set $C_i$ such that
\begin{itemize}
\item $C_i\subseteq\pcfsig(\col(i))\subseteq (\mu,\lambda_i)\cap\reg$,

\sk

\item $|C_i|<\sigma$, and

\sk

\item $\col(i)\subseteq\bigcup_{\psi\in C_i}B_\psi$.

\sk

\end{itemize}

Let $C = \bigcup_{i<\kappa}C_i$.   We need to note a few things.  First, the cardinality of $C$ is less than $\sigma$, as it is a
union of $\kappa$ sets, each of cardinality less than~$\sigma$.  Given our choice of $\theta$, this implies
\begin{equation}
|\pcf(C)|<\theta.
\end{equation}

For $\psi\in C$, we know $\max\pcf(B_\psi)=\psi<\lambda^+$, and that means that
\begin{equation}
(\forall^I \zeta<\theta)(\forall^J i<\kappa)\left[\tau^\zeta_i\notin B_\psi\right].
\end{equation}
This is easy to see:  if a set $X\subseteq A$ meets $\row(\zeta)$ on an $J$-positive set, then $\chi_\zeta\in\pcf(X)$.  If this
happens for an $I$-positive set of $\zeta<\theta$, then we get $\lambda^+\in \pcf(\pcf(X))=\pcf(X)$.   Said another way, if a
subset $X$ of $A$ is in $J_{<{\lambda^+}}[A]$, then for $I$-almost all rows $\zeta$, the set of $i<\kappa$ with $\tau^\zeta_i\in
X$ is in $J$.

Now comes the critical place where we use the fact that $|\pcf(C)|$ is less than $\theta$, the completeness of $I$.  For each $\psi\in\pcf(C)\cap\lambda$,
let
\begin{equation}
X_\psi = \{\zeta<\theta: \{i<\kappa: \tau^\zeta_i\in B_\psi\}\notin I\},
\end{equation}
that is, $X_\psi$ is the collection of $\zeta<\theta$ for which $\row(\zeta)$ meets $B_\psi$ on a $J$-large set.  Since
$\max\pcf(B_\psi)=\psi<\lambda^+$, it follows that
\begin{equation}
X_\psi\in I.
\end{equation}
Now $I$ is $\theta$-complete and $|\pcf(C)|<\theta$, so
\begin{equation}
\bigcup\{X_\psi:\psi\in\pcf(C)\cap\lambda\}\in I
\end{equation}
and therefore
\begin{equation}
(\forall^I\zeta<\theta)(\forall \psi\in\pcf(C)\cap \lambda)\left[\{i<\kappa: \tau^\zeta_i\in B_\psi\}\in J\right].
\end{equation}

Thus, we can fix $\zeta^*<\theta$ with the above property, which says informally that  $\row(\zeta^*)$ ``runs away''  from $B_\psi$ modulo $J$ for all $\psi<\lambda$ in $\pcf(C)$.

We are now ready to define our sequence $\langle \psi_i:i<\kappa\rangle$. First, we dispose of a triviality: if $i<\kappa$ is not
one of the places where $\sigma$ works for $\lambda_i$, we define $\psi_i$ to be $\mu^+$. Such an $i$ will be called a {\em
trivial coordinate}, and the set of trivial coordinates is in $J$.  If $i<\kappa$ is a non-trivial coordinate, then our construction will provide us with a $\psi_i$ in $C_i$ whose generator covers $\tau^{\zeta^*}_i$, that is such that 
\begin{equation}
\tau^{\zeta^*}_i\in B_{\psi_i}.
\end{equation}
Since $\psi_i$ is in $\pcfsig(\col(i))\subseteq\lambda_i$, we also know
\begin{equation}
\mu<\psi_i<\lambda_i
\end{equation}
whenever $i$ is a non-trivial coordinate.  Thus, the following claim will finish the proof of our lemma.

\begin{claim}
The sequence $\langle \psi_i:i<\kappa\rangle$ has the required properties, that is:
\begin{enumerate}
\item $\langle \psi_i:i<\kappa\rangle$ is a sequence of regular cardinals with supremum~$\lambda^*$,

\sk

\item $(\forall\lambda'<\lambda^*)(\forall^Ji<\kappa)\left[\lambda'\leq\psi_i\right]$,

\sk

\item $(\forall^J i<\kappa)\left[\psi_i<\lambda_i\right]$, and

\sk

\item $\tcf\left(\prod_{i<\kappa}\psi_i/ J\right)=\lambda^+$.

\sk

\end{enumerate}
\end{claim}

\begin{proof}[Proof of Claim]

First, note that if $i$ is a non-trivial coordinate, then
\begin{equation}
\tau^{\zeta^*}_i\leq\psi_i<\lambda_i,
\end{equation}
because $\tau^{\zeta^*}_i$ is in $B_{\psi_i}$.  Thus, we need only establish (4).

Let $D = \{\psi_i:i<\kappa\}$, and let $\mathcal{J}$ be the ideal on $D$ defined by
\begin{equation}
X\in\mathcal{J}\Longleftrightarrow \{i<\kappa:\psi_i\in X\}\in J.
\end{equation}
The statement (4) is equivalent to proving that the true cofinality of $\prod D/ \mathcal{J}$ is defined and equal to
$\lambda^+$, and this follows provided we can establish
\begin{equation}
J_{<{\lambda^+}}[D]\subseteq\mathcal{J}.
\end{equation}
Pcf theory tells us that this is equivalent to showing
\begin{equation}
\label{goal}
\{i<\kappa: \psi_i\in B_\psi\}\in J
\end{equation}
whenever $\psi\in\pcf(D)\cap\lambda$, as this tells us that $B_\psi\in\mathcal{J}$ for each such $\psi$.  Clearly we need only
worry about the non-trivial coordinates.

Given such a $\psi\in\pcf(D)\cap \lambda$, we note that $\psi\in\pcf(C)\cap\lambda$ and so $\row(\zeta^*)$ runs away from
$B_\psi$ by choice of $\zeta^*$, that is,
\begin{equation}
\label{part1}
\{i<\kappa: \tau^{\zeta^*}_i\in B_\psi\}\in J.
\end{equation}
Our generators were chosen to be transitive, so that
\begin{equation}
\psi_i\in B_\psi\Longrightarrow B_{\psi_i}\subseteq B_\psi,
\end{equation}
and since $\tau^{\zeta^*}_i\in B_{\psi_i}$, we conclude
\begin{equation}
\label{part2}
\psi_i\in B_{\psi}\Longrightarrow \tau^{\zeta^*}_i\in B_\psi.
\end{equation}
Putting (\ref{part1}) and (\ref{part2}) together establishes (\ref{goal}), and we're done establishing both the claim, and Lemma~\ref{mainlemma}.
\end{proof}

\subsection{Proof of the main theorem}

Recall that we are assuming that $\lambda$ is the least counterexample for $\mu$, and we know that $\lambda$ is singular with
$\cf(\lambda)=\cf(\mu)$.  We apply Lemma 3  to find a regular cardinal $\kappa$ such that $\cf(\mu)<\kappa<\mu$, a
$\kappa$-complete ideal $J$ on $\kappa$, and singular $\lambda^*$ with $\mu<\lambda^*<\lambda$ such that there is an admissible
sequence $\bar{\lambda}$ representing $\lambda^+$ at $\lambda^*$ modulo $J$. Since $\kappa$ is uncountable and $J$ is
$\kappa$-complete, we can find such a $\bar{\lambda}$ that is {\em minimal} modulo $J$, that is, we can choose
$\bar{\lambda}=\langle \lambda_i:i<\kappa\rangle$ such that

\begin{itemize}
\item each $\lambda_i$ is a regular cardinal between $\mu$ and $\lambda^*$

\sk

\item $\lambda^* = \sup\{\lambda_i:i<\kappa\}$

\sk

\item $(\forall\lambda'<\lambda^*)(\forall^Ji<\kappa)[\lambda'\leq\lambda_i]$

\sk

\item $\lambda^+=\tcf\left(\prod_{i<\kappa}\lambda_i/J\right)$, and

\sk

\item if $\bar{\psi}=\langle\psi_i:i<\kappa\rangle$ is an admissible sequence below $\bar{\lambda}$, then $\bar{\psi}$ cannot
    represent $\lambda^+$.
\end{itemize}

In light of Lemma~\ref{mainlemma}, the minimality of $\bar{\lambda}$ means that no single $\sigma$ can work for all (even almost all)
$\lambda_i$.  However, since $\kappa>\cf(\mu)$ and $J$ is $\kappa$-complete, we can find $\langle
X_\alpha:\alpha<\cf(\mu)\rangle$ and $\langle \sigma_\alpha:\alpha<\cf(\mu)\rangle$ such that
\begin{itemize}
\item $\langle X_\alpha:\alpha<\cf(\mu)\rangle$ is a partition of $\kappa$ into $J$-positive sets,

\sk

\item $\langle \sigma_\alpha:\alpha<\cf(\mu)\rangle$ is an increasing sequence of regular cardinals cofinal in $\mu$

\sk

\item $\sigma_\alpha$ works for $\lambda_i$ whenever $i\in X_\alpha$.

\sk

\end{itemize}

For $\alpha<\cf(\mu)$, we let $J_\alpha$ be the ideal $J+(\kappa\setminus X_\alpha)$, that is,
\begin{equation}
X\in J_\alpha\Longleftrightarrow X\cap X_\alpha\in J.
\end{equation}

Given $\alpha<\cf(\mu)$, we note that the hypotheses of Lemma 4 are satisfied with $J = J_\alpha$ and $\sigma = \sigma_\alpha$,
so we get a corresponding sequence of cardinals $\langle \psi_i^\alpha:\alpha<\kappa\rangle$ that represent $\lambda^+$ modulo
$J_\alpha$.  Let $\bar{f}^\alpha = \langle f_\epsilon^\alpha:\epsilon<\lambda^+\rangle$ witness
\begin{equation}
\tcf\left(\prod_{i<\kappa}\psi_i^\alpha/ J_\alpha \right)=\lambda^+.
\end{equation}

Now we try the obvious thing and things work because $\cf(\mu)<\kappa$ and $J$ is $\kappa$-complete.
Given $i<\kappa$, we let $\alpha(i)$ be the unique $\alpha<\cf(\mu)$ with $i\in X_\alpha$, and define for $i<\kappa$
\begin{equation}
\psi_i = \psi^{\alpha(i)}_i.
\end{equation}

\begin{claim} $\langle\psi_i:i<\kappa\rangle$ is an admissible sequence below $\bar{\lambda}$ that
represents~$\lambda^+$.
\end{claim}

\begin{proof}
Certainly each $\psi_i$ is a regular cardinal between $\mu$ and $\lambda^*$.  Given $\lambda'<\lambda^*$ and $\alpha<\cf(\mu)$,
the definition of $J_\alpha$ tells us that
\begin{equation}
Y_\alpha = \{i\in X_\alpha: \lambda'\leq\psi^\alpha_i\}\in J
\end{equation}
and since $J$ is $\kappa$-complete,
\begin{equation}
\{i\in X: \lambda'\leq\psi_i\} = \bigcup_{\alpha<\cf(\mu)}\{i\in X_\alpha: \lambda'\leq\psi_i^\alpha\}\in J.
\end{equation}
Similarly, $\{i<\kappa: \psi_i<\lambda_i\}$ is in $J$, so our sequence is below $\bar{\lambda}$.

Now given $\epsilon<\lambda^+$, define a function $f_\epsilon$ with domain $\kappa$ by
\begin{equation}
f_\epsilon(i) = f^{\alpha(i)}(i),
\end{equation}
and the sequence $\langle f_\epsilon:\epsilon<\lambda^+\rangle$ shows us that $\bar{\psi}$ represents $\lambda^+$ at $\lambda^*$
modulo $J$.
\end{proof}

The preceding claim now yields our contradiction, as $\bar{\lambda}$ was allegedly a $J$-minimal $J$-admissible sequence representing $\lambda^+$ at $\lambda^*$.

\end{proof}

\end{document}